\documentclass{amsart}

\usepackage{graphicx}
\usepackage{amsmath}     
\usepackage{amsthm}
\usepackage{xypic}

\newtheorem{theorem}{Theorem}[section]
\newtheorem{lemma}[theorem]{Lemma}
\newtheorem{proposition}[theorem]{Proposition}
\newtheorem{corollary}{Corollary}[section]
\newtheorem{remark}[theorem]{Remark}
\theoremstyle{definition}
\newtheorem{definition}[theorem]{Definition}

\numberwithin{equation}{section}

\def\L{\mathcal{L}}

\def\dr{\ar@{->}[r]}

\renewcommand{\theequation}{\thesection.\arabic{equation}}
\def\X{\mathcal{X}}

\def\Ext{\mbox{Ext}}

\def\Hom{\mbox{Hom}}

\def\Ker{\mathsf{Ker}\hspace{.01in}}

\newcommand{\ann}{\mathsf{ann}}
\newcommand{\gl}{\mathsf{l.gl.dim}\hspace{.01in}}
\newcommand{\pd}{\mathsf{pd}}

\def\Mod{\mathsf{Mod}\hspace{.01in}}

\def\Tor{\mbox{Tor}}

\def\dim{\mathsf{dim_{K}}}

\def\sup{\mathsf{sup}}

\def\cd{\mathsf{cd}}
\def\wpd{\mathsf{wpd}}

\begin{document}

\title{Tensor Products of Flat Cotorsion Modules and Cotorsion Dimension}

\author{Hu Yonggang}
\address{China Foreign Affairs University, Xicheng District, Beijing, 100037}
\email{huyonggang@cfau.edu.cn}


\author{Ma Linyu}
\address{China University of Mining and Technology (Beijing), School of Science, Haidian District, Beijing, 100083}
\email{2466352848@qq.com}
\thanks{The first author was supported by National Natural Science Foundation of China (Grant  No. 12401049) and Shanxi Provincial Basic Research Program (Free Exploration) Youth Scientific Research Project (Grant  No. 202403021212001). The author second was supported by the Fundamental Research Funds for the Central Universities (No. 2024ZKPYLX01) and the Natural Science Foundation of China (Grant Nos. 12271321).}
\author{Wang Xintian}
\address{China University of Mining and Technology (Beijing), School of Science, Haidian District, Beijing, 100083}
\email{201803@cumtb.edu.cn}





\begin{abstract}
This paper studies the tensor product of flat cotorsion modules. Let~$R$~and $S$ be~$k$-algebras. We prove that both~$R$-module\ $M$ and~$S$-module\ $N$ are flat cotorsion modules if and only if~$M\otimes_{k}  N$ is a flat cotorsion~$R\otimes_{k} S $-module. Based on this conclusion, we provide a lower bound for the global cotorsion dimension of the tensor product algebra~$R\otimes_{k}S $ under appropriate conditions.
\end{abstract}

\maketitle



\section{ introduction}

Cotorsion pairs are important concepts in abelian categories and triangulated categories. First appeared in the literature~\cite{Salce}, they are closely related to model structures~\cite{Hovey01,Hovey02} and have significant applications in homological algebra and representation theory, such as in~\cite{Krause,Angeleri}. A cotorsion pair consists of two subcategories that are orthogonal with respect to the Ext functor. \par
The flat cover conjecture was first proposed by\ Enochs~\cite{Enoch81}. This conjecture states that every module has a flat cover. Enochs also proved that this is equivalent to every module having a flat precover. Let~$R$~ be a ring, and $\X$ be a class of~$R$-modules. An~$\X$-precover of an~$R$-module~$M$~\cite{Enoch81} is a morphism~$ f :X\rightarrow M$ with~$X\in\X$, such that for any morphism $g : Y \rightarrow M$ with~$Y\in\X$, there exists a morphism~$ h : Y \rightarrow X$ satisfying~$f h = g$. Dually, one can define an~$\X$-preenvelope. Cotorsion modules, introduced by Enochs in~\cite{Enochs:2}, generalize the concept of cotorsion groups defined by~Harrison~\cite{Harrison} and Fuchs~\cite{Fuchs}. Enochs' definition of the cotorsion modules differs from that of Matlis in~\cite{Matlis}. Bican,~El~Bashir, and~Enochs~proved in~\cite{Bican}~that the class of flat modules and the class of cotorsion modules form a complete cotorsion pair. This result completely resolves the famous ``flat cover conjecture". Moreover, the core of the flat cotorsion pair, i.e.the intersection of the flat class and the cotorsion class (also known as flat cotorsion modules)-exhibits good homological properties~\cite{Nakamura}, similar to projective and injective objects. In this context, Mao and Ding introduced the concept of cotorsion dimension in~\cite{Mao:2}, whitch measures how far a ring is from being a perfect ring. In~\cite{Mao:2}, a lower bound for the cotorsion dimension of a ring is given, and the change of cotorsion dimension under ring extensions is discussed.\par
The study of homological dimensions in tensor product algebras has a long history. As early as the 1950s,  Eilenberg, Auslander, and Nakayama published a series of works on the homological dimensions of tensor product algebras in the  journal \textit{Nagoya Mathematical Journal}~\cite{Auslander1,Auslander2,Eilenberg1,Eilenberg2}. Eilenberg considered the problem of homological dimensions over tensor product algebras to be ``surprisingly difficult''~\cite{Eilenberg2}. Despite this, many important conclusions about the homology and representation theory of tensor product algebras have been established. For instance, in recent years, the higher-order representation theory of tensor product algebras has attracted scholars' attention~\cite{Pasquali,Pasquali2}. In their foundational contribution to homological dimensions, Eilenberg, Rosenberg, and Zelinsky established critical inequalities for projective dimensions in tensor products: Let $M$ and $N$ be modules over $k$-algebras $R$ and $S$. The projective dimension of their tensor product satisfies $$\pd(M)+\wpd(N)\leq \pd(M\otimes N)\leq \pd(M)+\pd(N),$$ where $\wpd$
and $\pd$ denote the flat dimension and projective dimension, respectively. Zhou further proved in \cite{Zhou:3} that under certain conditions on $k$-algebras $R$ and $S$, the inequality $\gl(R \otimes S) \geq \gl(R) + \gl(S)$ holds, where $\gl$ denotes the left global dimension of the algebra. Since flat-cotorsion modules share similar properties with projective-injective modules, it is natural to consider the cotorsion dimension of tensor product algebras. The main results of this paper are as follows:

\textbf{Theorem A}: Let $M$ be an $R$-module and $N$ be an $S$-module. Then $M \otimes N$ is a flat-cotorsion $R \otimes S$-module if and only if both $M$ and $N$ are flat-cotorsion modules.

\textbf{Theorem B}: Let $M$ be an $R$-module and $N$ be an $S$-module. The following statements are equivalent:
\begin{itemize}
  \item[\upshape{(P1)}] If~$\cd(M)=1$,~$\cd(N)=1$, then~ $\cd(M\otimes N)= 2$.
  \item[\upshape{(P2)}]If~$\cd(M)=m$,~$\cd(N)=n$, where~$m,~n\geq 1$ is a positive integer, then~ $\cd(M\otimes N)= n+m$.
\end{itemize}

In particular, when (P1)  was established, then~$l.cot.D(R\otimes S)\geq l.cot.D(R)+l.cot.D(S)$.\par
The structure of this paper is as follows. The second section recalls related concepts and important conclusions. In section 3, the tensor product property of cotorsion is discussed. In section 4, we study the cotorsion dimension of tensor product algebras.

\par

\section{ preliminaries}

This article always assumes that\ $k$ is a field, \ $R$ and\ $S$ are\ $k$-algebras.  All modules are left modules unless otherwise specified.

 Any two\ $k$-algebras\ $R$ and $S$, the tensor product\ $R\otimes_k S$ is a\ $k$-module equipped with a multiplication$$ (r_1\otimes s_1)(r_2\otimes s_2)=(r_1r_2)\otimes(s_1s_2),$$
which converts\ $R\otimes_k S$ into a\ $k$-algebras. For the convenience of writing, we always write\ $R\otimes_k S$ briefly as\ $R\otimes S$.

Let\ $R$,\ $S$ be\ $k$-algebras. Arbitrary\ $\Lambda$-modules\ $X$ is naturally\ $k$-modules. It is often necessary to consider that\ $X$ is a\ $S$-module, in this case assuming:
\begin{itemize}
  \item[(1)]    The operators of\ $R$ and\ $S$ on\ $X$ is commute, i.e.$r(s x)$=$s(r x)$;
  \item[(2)]  The structure of\ $X$ as a\ $k$-module induced by\ $R$ is the same as that induced by\ $S$.
\end{itemize}

In particular, suppose~$X$  is a left\ $R$-module and a left\ $S$-module~(situation\ $_{\Lambda-\Gamma}X)$. Define~$R\otimes S$-module on~$X$
$$(\lambda\otimes\gamma)a=\lambda(\gamma a)=\gamma(\lambda a).$$
Then\ $X$ is a left\ $R\otimes S$-module.
The converse is also clear: every left\ $R\otimes S$-module is obtained in this way from a unique\ $R$-$S$-module. The same applies with ``left" replaced by ``right".

 If\ $R$ and\ $S$ are\ $k$-algebras, for any\ $a\in R$, $b\in S$, there exists a twist map\ $t: R\otimes S\rightarrow S\otimes R$, such that\ $t(a\otimes b)=b\otimes a$. In fact, $t$ is an isomorphism of \ $k$-algebras.
\begin{proposition}{\cite[Lemma 2,3]{Zhou:3}}\label{prop0}
Let
\begin{equation}\label{1-1}
  0\longrightarrow X \xrightarrow{d_{0}} C_{0} \xrightarrow{d_{1}} C_{1}\longrightarrow 0
\end{equation}
be a short exact sequence in\ $\Mod R$, and
\begin{equation}\label{1-1}
  0\longrightarrow Y \xrightarrow{\partial_{0}} D_{0} \xrightarrow{\partial_{1}} D _{1}\longrightarrow 0
\end{equation}
be a short exact sequence in\ $\Mod R$. Then there exists a short exact sequence in\ $\Mod R\otimes S$
\begin{equation}\label{1-1}
  0\longrightarrow X\otimes Y \xrightarrow{d_{0}\oplus\partial_{0}} C_{0}\otimes D_{1}\oplus C_{1}\otimes D_{0} \xrightarrow{\partial_{0}-d_{0}} C_{0}\otimes D _{0}\xrightarrow{d_{1}\otimes\partial_{1}} C_{1}\otimes D _{1}\longrightarrow0
\end{equation}
where
\begin{align*}
  d_{0}\oplus\partial_{0}:~& X\otimes Y\longrightarrow C_{0}\otimes D_{1}\oplus C_{1}\otimes D_{0} \\
   & x\otimes y\longmapsto d_{0}(x)\otimes y\oplus x\otimes \partial_{0}(y)
\end{align*}
and
\begin{align*}
  \partial_{0}-d_{0}:~ & C_{0}\otimes D_{1}\oplus C_{1}\otimes D_{0}\longrightarrow  C_{0}\otimes D_{0}\\
   & x'\otimes y\oplus x\otimes y'\longmapsto x'\otimes \partial_{0}(y)- d_{0}(x')\otimes y
\end{align*}
are both homomorphisms of \ $R\otimes S$-modules; $\Ker d_{1}\otimes\partial_{1}=C_{0}\otimes D_{1}+ C_{1}\otimes D_{0}$.
\end{proposition}

\begin{proposition}{ {\rm\cite[\text{Chapter 5, Section 5, Corollaries 1 and 2}]{Zhou:4}}}\label{p1}
Let  \ $R$ and \ $S$ be\ $k$-algebras. Suppose that\ $M$ is\ $R$-module, and\ $N$ is\ $S$-module.
\begin{itemize}
      \item[\upshape{(1)}] If\ $M$ is a projective\ $R$-module  and \ $N$ is a projective\ $S$-module, then\ $M\otimes N$ is a projective\ $R\otimes S$-module;
      \item[\upshape{(2)}]If\ $M$ is a flat\ $R$-module  and \ $N$ is a flat\ $S$-module, then\ $M\otimes N$ is a flat\ $R\otimes S$-module.
    \end{itemize}
\end{proposition}
\begin{proposition}{\upshape{\cite{Cartan:2}}}\label{p2}
 Let\ $\Lambda,~\Gamma,~\Sigma$ be\ $k$-projective\ $k$-algebras. In the situation

 $(A_{\Lambda-\Gamma},_{\Lambda}B_{\Sigma} ,_{\Gamma-\Sigma}C)$,
assume
{\rm$$\Tor_n^\Lambda(A,B)=0=\Tor_n^\Sigma(B,C)$$} for\ $n>0$.
Then there is an isomorphism
{\rm$$\Tor_{n}^{\Gamma\otimes\Sigma}(A\otimes_\Lambda B,C)\cong \Tor_{n}^{\Lambda\otimes\Gamma}(A,B\otimes_\Sigma C)$$}
for\ $n>0$.
\end{proposition}
\begin{proposition}{\upshape{\cite{Chen:1,Pasquali}}}\label{p3}
 If $M, N\in\Lambda,\ X ,Y\in\Gamma $, then there is a function isomorphism
{\rm$$\Ext_{\Lambda\otimes \Gamma}^n(M\otimes_k N,X\otimes_kY)\cong\bigoplus_{p+q=n}\Ext_\Lambda^p(M,X)\otimes\Ext_\Gamma^q(N,Y).$$}
\end{proposition}
\begin{definition}{\upshape{\cite{Enochs:2}}}\label{d1}
A module M is said to be cotorsion if\ $\Ext_R^1(F,M)=0$ for all flat module\ $F$. If~$M$ is both a flat module and a cotorsion module, then~$M$ is called a flat cotorsion module.
\end{definition}
\begin{definition}{\upshape{\cite{Mao:2}}}\label{d2}
Let\ $R$ be a~$k$-algebra. For any right\ $R$-module\ $M$, the cotorsion dimension $\cd(M)$ of $M$ is defined to be the smallest integer\ $n\geq0$ such that\ $\Ext_R^{n+1}(F,M)=0$ for any flat right\ $R$-module\ $F$. If there is no such\ $n$, set\ $\cd(M)=\infty$.

The left global cotorsion dimension of the algebra~$R$ is defined as the supremum of the cotorsion dimensions of~$R$-modules, denoted as~$l.cot. D(R)$.
\end{definition}

\begin{proposition}{\upshape{\cite{Mao:2}}}\label{p5}
For any right~$R$-module~$M$ and integer ~$n\geq0$, the following are equivalent:
\begin{itemize}
  \item[\upshape{(1)}]$\cd(M)\leq n$;
  \item[\upshape{(2)}] {\rm$\Ext_R^{n+1}(F,M)=0$}, for any flat~$R$-module~$F$;
  \item[\upshape{(3)}] {\rm$\Ext_R^{n+j}(F,M)=0$}, for any flat~$R$-module~$F$ and~$j\geq 1$;
  \item[\upshape{(4)}] If the sequence~$0\rightarrow M\rightarrow C^0\rightarrow C^1\rightarrow\cdots\rightarrow C^{n-1}\rightarrow C^n\rightarrow0$ is exact with ~$C^0,~C^1,~\cdots,~ C^{n-1}$~cotorsion, then~$C^n$  is also a cotorsion module.
\end{itemize}
\end{proposition}
\begin{remark}{\upshape{\cite{Mao:2}}}\label{remark2-1}
 Let $M$ be any right~$R$-module. The the following are equivalent:
\begin{itemize}
  \item[(1)] $\cd(M)=\inf\{n~|~\textrm{there exists an exact sequence}~0\rightarrow M\rightarrow C^{0}\rightarrow C^{1}\rightarrow \cdots\rightarrow C^{n}\rightarrow 0,~\textrm{where}~C^{i}~ \textrm{are a cotorsion modules}\}$
  \item[(2)] The integer $n$ such that $M$ admits a minimal cotorsion resolution, i.e. an exact sequence
  $$0\rightarrow M\rightarrow C^{0}\xrightarrow{d^{0}} C^{1}\xrightarrow{d^{1}} \cdots\rightarrow C^{n}\xrightarrow{d^{n}}\cdots$$
  where each\ $C^i$ is cotorsion,\ $L^i=Coker(C^{i-2}\rightarrow C^{i-1})\rightarrow C^i$ is a cotorsion envelope of\ $L^i, C^i\neq0,\ i=0,1,...,n,\ C^{-2}=0.\ C^-1=M$.
\end{itemize}

\end{remark}
\begin{proposition}{\upshape{\cite{Mao:2}}}\label{p6}
Let~$R$ be a\ $k$-algebra, and~$0\rightarrow X\rightarrow Y\rightarrow Z\rightarrow0$~be an exact sequence of\ $R$-modules. If two of~$\cd(X),~\cd(Y),~\cd(Z)$ are finite, so is the third. Moreover
\begin{itemize}
  \item[\upshape{(1)}] $\cd(Y)\leq \sup\{\cd(X),~\cd(Z)\}$;
  \item[\upshape{(2)}] $\cd(X)\leq \sup\{\cd(Y),~\cd(Z)+1\}$;
  \item[\upshape{(3)}] $\cd(Z) \leq \sup\{\cd(Y),~\cd(X)-1\}$.
\end{itemize}

\end{proposition}

\begin{proposition}\label{prop2-7}
Let~$R$ be a\ $k$-algebra. If
\begin{equation}\label{eq1-2}
0\rightarrow X\rightarrow Y\rightarrow Z\rightarrow0
\end{equation}
is a short exact sequence of\ $R$-modules, then
\begin{itemize}
  \item[\upshape{(1)}] If\ $\cd(Y)>\cd(Z)$, then\ $\cd(X)=\cd(Y)$;
  \item[\upshape{(2)}] If\ $\cd(Y)<\cd(Z)$, then\ $\cd(X)=\cd(Z)+1$;
  \item[\upshape{(3)}] If\ $\cd(Y)=\cd(Z)$, then\ $\cd(X)\leq\cd(Z)+1$.
\end{itemize}

\end{proposition}
\begin{proof}
For any flat~$R$-module~$F$, applying the functor\ $\Hom_{R}(F, -)$ to the short exact sequence\ (\ref{eq1-2}), we obtain a long exact sequence
\begin{equation}\label{eq1-3}
\begin{split}
 \xymatrix@R=2em{\cdots\ar[r]&\Ext^{n}(F,X)\ar[r]&\Ext^{n}(F,Y)\ar[r]&\Ext^{n}(F,Z)\ar@{-}[r]+<1.3mm,0mm>&
\ar  `r[d] `[l]
`[lllld] `[dlll] [dlll]\\&
\Ext^{n+1}(F,X) \ar[r]&\Ext^{n+1}(F,Y)\ar[r]&\Ext^{n+1}(F,Z)\ar[r]&\cdots~;}
\end{split}
\end{equation}

(1) Suppose~$\cd(Z)=n$. When~$m\geq n$, $\Ext^{m}(F,Z)=0$, but~$\Ext^{n}(F,Y)\neq0$, then~$\Ext^{n}(F,X)\neq0$. Thus, from the exact sequence~(\ref{eq1-3}), when\ $j>0$, we have
$$\Ext^{n+j}(F,X)\cong \Ext^{n+j}(F,X).$$
Therefore, $\cd(X)=\cd(Y)$.

(2) Suppose~$\cd(Y)=n$. When~$m\geq n$, $\Ext^{m}(F,Y)=0$, but~$\Ext^{n}(F,Z)\neq0$, then~$\Ext^{n+1}(F,X)\neq0$. Thus, from the exact sequence~(\ref{eq1-3}), when\ $j>0$, we have
$$\Ext^{n+j}(F,Z)\cong \Ext^{n+j+1}(F,X).$$
Therefore, $\cd(X)=\cd(Z)+1$.

(3) Suppose~$\cd(Y)=\cd(Z)=n$. When~$m\geq n$, $\Ext^{m}(F,Y)=\Ext^{m}(F,Z)=0$, then~$\Ext^{n+1}(F,X)=0$.

According to Proposition~\ref{p5}, we can conclude that $\cd(X)\leq\cd(Z)+1$.
\end{proof}
\section{TENSOR PRODUCTS OF COTORSION MODULES}

\begin{lemma}\label{l1}
Let $F$ be an $R$-module and $H$ be an $S$-module. Then $H \otimes F$ is a flat $S \otimes R$-module if and only if $F$ is a flat $R$-module and $H$ is a flat $S$-module.
\end{lemma}
\begin{proof}
Necessity: Assume conversely: If $H$ is not a flat $S$-module, there exists a right $S$-module $N$ such that $\Tor_1^S(N,H) \neq 0$. Applying Proposition~\ref{p2} to the triple $(N_{k \otimes S}, {}_{k}R_R, {}_{S-R}(H \otimes_k F))$, since
$$\Tor_1^k(N,R) = 0 = \Tor_1^R(R, H \otimes_k F),$$
we have:
\begin{align*}
   & \Tor_1^{S \otimes R}(N \otimes R, H \otimes F) \\
   & \cong \Tor_1^{k \otimes S}(N, R \otimes_R (H \otimes F)) \\
   & \cong \Tor_1^{k \otimes S}(N, (R \otimes_R F) \otimes H) \\
   & \cong \Tor_1^{k \otimes S}(N, F \otimes H) \\
   & \cong \Tor_1^{k \otimes S}(N, \oplus_\lambda H) \quad (\text{where } \lambda = \dim F) \\
   & = \oplus_\lambda \Tor_1^{S}(N,H) \neq 0.
\end{align*}

Since $H \otimes F$ is a flat $S \otimes R$-module, $\Tor_1^{S \otimes R}(N \otimes R, H \otimes F) = 0$, leading to a contradiction. Hence, $H$ is flat.

Since~$S\otimes R\cong R\otimes S$ and~$H\otimes F\cong F\otimes H$, \ $F\otimes H$ is a flat~$R\otimes S$-module. Similarly, $F$ is flat by symmetry.

Sufficiency: Follows directly from Proposition~\ref{p1}.
\end{proof}

\begin{lemma}\label{l2}
If $M$ is a flat $R$-module, $N$ is a flat $S$-module, $M_1$ is a cotorsion $R$-module, and $N_1$ is a cotorsion $S$-module, then for any $n \geq 1$,
$$\Ext_{R \otimes S}^n(M \otimes N, M_1 \otimes N_1) = 0.$$
\end{lemma}
\begin{proof}
By Proposition~\ref{p3},
$$\Ext_{R \otimes S}^n(M \otimes N, M_1 \otimes N_1) \cong \bigoplus_{p+q=n} \Ext_R^p(M, M_1) \otimes \Ext_S^q(N, N_1).$$

For $p + q = n \geq 1$, at least one of $p$ or $q$ is positive. Since $\Ext_R^p(M, M_1) = 0$ or $\Ext_S^q(N, N_1) = 0$, the result follows.
\end{proof}

\begin{lemma}\label{l3}
If $F$ is a flat $R \otimes S$-module, then $F$ is both a flat $R$-module and a flat $S$-module.
\end{lemma}
\begin{proof}
For any right $R$-module $X$, consider the triple $(X_{k-R}, {}_{k}S_S, {}_{R-S}F)$, and
$$\Tor_i^R(X, S) = 0 = \Tor_i^S(S, F)$$
Proposition~\ref{p2} gives:
$$\Tor_i^{R \otimes S}(X \otimes_k S, F) \cong \Tor_i^R(X, F).$$
As $F$ is flat over $R \otimes S$, $\Tor_i^R(X, F) = 0$. Hence, $F$ is flat over $R$. Similarly for $F$ is a flat $S$-module.
\end{proof}

\begin{theorem}\label{t1}
Let $M$ be an $R$-module and $N$ be an $S$-module. Then $M \otimes N$ is a cotorsion $R \otimes S$-module if and only if both $M$ and $N$ are cotorsion modules.
\end{theorem}
\begin{proof}
Sufficiency: Let $F$ be any flat $R \otimes S$-module. Using the twist isomorphism $t: R \otimes S \to S \otimes R$, $F$ becomes a flat $S \otimes R$-module. The definition of the modules operation is given as~$(s\otimes r)\ast f=t^{-1}(s\otimes r)f=( r\otimes s)f$. Next, we define the mapping
$$\pi: S \otimes F \to F, \quad s \otimes f \mapsto sf.$$
Obviously,\ $\pi$ is well-defined. We prove that\ $\pi$ is a left\ $S\otimes R$-module homomorphism. For any\ $ \alpha\otimes\beta\in S\otimes R,\ s\otimes f\in S\otimes F$, we have
\begin{align*}
   & \pi((\alpha\otimes\beta)(s\otimes f)) \\
   & =\pi(\alpha s\otimes\beta f)\\
   & =(\alpha s)(\beta f)\\
   & =\beta(\alpha sf)\\
   & =(\beta\otimes\alpha)(sf)\\
   & =t^{-1}(\alpha\otimes\beta)(sf)\\
   & =t^{-1}(\alpha\otimes\beta)\pi(s\otimes f)\\
   & =(\alpha\otimes\beta)\ast\pi(s\otimes f).
\end{align*}
Since \ $\pi(1\otimes f)=1f=f$ for any\ $ f\in F$,\ $\pi$ is surjective.
$\Ker \pi = \ann_S(F) \otimes F$, where\ $\ann_S(F)=\{s\in S|~sf=0,~f\in F\}$ is the annihilator of\ $F$ in \ $S$.
Therefore, there exists a short exact sequence of $S \otimes R$-modules:
\begin{equation}\label{eq2-1}
    0 \rightarrow \ann_S(F) \otimes F \xrightarrow{\eta \otimes \text{id}_{F}} S \otimes F \rightarrow F \rightarrow 0
\end{equation}
where $\eta: \ann_S(F) \rightarrow S$ is the inclusion map.
Notice that from the short exact sequence \eqref{eq2-1}, there exists a short exact sequence of $R \otimes S$-modules:
\begin{equation}\label{eq2-2}
    0 \rightarrow F \otimes \ann_S(F) \xrightarrow{\text{id}_{F} \otimes \eta} F \otimes S \rightarrow F \rightarrow 0.
\end{equation}
Notice also that there is a short exact sequence of $S$-modules:
\begin{equation}\label{eq2-3}
    0 \rightarrow \ann_S(F) \xrightarrow{\eta} S \rightarrow S / \ann_S(F) \rightarrow 0.
\end{equation}
Applying the exact functor $F \otimes -$ to the short exact sequence \eqref{eq2-3}, we obtain:
\[
    0 \rightarrow F \otimes \ann_S(F) \xrightarrow{\text{id}_{F} \otimes \eta} F \otimes S \rightarrow F \otimes S / \ann_S(F) \rightarrow 0.
\]
Thus, there exists a commutative diagram of short exact sequences of $R \otimes S$-modules.
$$\xymatrix{
  0 \ar[r] & F\otimes\ann_S(F) \ar@{=}[d] \ar[r] &F\otimes S \ar@{=}[d] \ar[r] & F\otimes S/\ann_S(F) \ar@{-->}[d]^{\theta} \ar[r]&0 \\
  0 \ar[r] & F\otimes\ann_S(F)  \ar[r] &F\otimes S \ar[r] & F  \ar[r]&0  }$$

By the universal property of quotients, we obtain the map $\theta$ in the third column of the diagram. Then, by the Five Lemma, $\theta$ is an $R \otimes S$-module isomorphism. Consequently, we have the isomorphic extension group:
\begin{align}\label{eq2-6}
  \Ext_{R \otimes S}^1(F, M \otimes N) &\cong \Ext_{R \otimes S}^1(F \otimes S / \ann_S(F), M \otimes N)
\end{align}
Since $F \otimes S / \ann_S(F) \cong F$ is a flat $R \otimes S$-module, according to Lemma~\ref{l1}, $S / \ann_S(F)$ is a flat $S$-module. By Lemmas~\ref{l2}, Lemma \ref{l3}, and the isomorphism~(\ref{eq2-6}), we conclude that $\Ext_{R \otimes S}^1(F, M \otimes N) = 0$. Therefore, $M \otimes N$ is a cotorsion $R \otimes S$-module.

\textbf{``Necessity''}  Suppose $M$ is not a cotorsion $R$-module. Then there exists a flat $R$-module $F_0$ such that
\[
\Ext_R^1(F_0, M) \neq 0.
\]
According to Proposition~\ref{p3}, we have
\begin{align*}
  \Ext_{R \otimes S}^1(F_0 \otimes S, M \otimes N) &\cong \Ext_R^1(F_0, M) \otimes \Ext_S^0(S, N) \oplus \Ext_R^0(F_0, M) \otimes \Ext_S^1(S, N) \\
  &\cong \Ext_R^1(F_0, M) \otimes \Hom_S(S, N) \\
  &\cong \Ext_R^1(F_0, M) \otimes N.
\end{align*}
Since $\Ext_R^1(F_0, M) \neq 0$ and $N \neq 0$, it follows that $\Ext_R^1(F_0, M) \otimes N \neq 0$. However, $F_0 \otimes S$ is a flat $R \otimes S$-module, and $M \otimes N$ is a cotorsion module. Hence, $\Ext_{R \otimes S}^1(F_0 \otimes S, M \otimes N) = 0$. This leads to a contradiction. Therefore, $M$ is a cotorsion $R$-module. Similarly, it can be shown that $N$ is a cotorsion $S$-module.
\end{proof}
\begin{remark}
Notably, while injective modules are special cases of cotorsion modules, Theorem~\ref{t1} is not limited to injectives, as tensor products of injectives need not be injective (see \cite[p298]{Zhou:4}).
\end{remark}
\begin{corollary}
Let $M$ be an $R$-module and $N$ be an $S$-module. Then $M \otimes N$ is a flat cotorsion $R \otimes S$-module if and only if both $M$ and $N$ are flat cotorsion modules.
\end{corollary}
\section{propositions on tensor product of cotorsion dimensions}

\begin{proposition}\label{prop4-0}
Let $M$ be an $R$-module, $N$ be an $S$-module, and $m, n$ be positive integers. The following statements hold:
\begin{itemize}
\item[\upshape{(1)}] If $M$ is a cotorsion module and $\cd(N) = n$, then $\cd(M \otimes N) = n$.
\item[\upshape{(2)}] If $\cd(M) = m$ and $N$ is a cotorsion module, then $\cd(M \otimes N) = m$.
\end{itemize}
\end{proposition}
\begin{proof}
(1) Since $\cd(N) = n$, we obtain the shortest cotorsion resolution of $N$:
\begin{equation}\label{eq4-00}
0 \rightarrow N \xrightarrow{\partial^{0}} D^0 \xrightarrow{\partial^{1}} D^1 \rightarrow \cdots \rightarrow D^{n-1} \xrightarrow{\partial^{n}} D^n \rightarrow 0,
\end{equation}
where $\Ker \partial_i$ is not cotorsion modules for $0 < i \leq n$. Applying the exact functor $M \otimes -$ to the exact sequence (\ref{eq4-00}), we obtain an exact sequence:
\begin{equation}\label{eq4-01}
0 \rightarrow M \otimes N \xrightarrow{\text{id}_M \otimes \partial^{0}} M \otimes D^0 \xrightarrow{\text{id}_M \otimes \partial^{1}} M \otimes D^1 \rightarrow \cdots \rightarrow M \otimes D^{n-1} \xrightarrow{\text{id}_M \otimes \partial^{n}} M \otimes D^n \rightarrow 0.
\end{equation}
Note that $\Ker(\text{id}_M \otimes \partial^i) = M \otimes \Ker \partial^i$ is not cotorsion modules for $0 < i \leq n$. Since $M$ is a cotorsion module, Theorem \ref{t1}  implies that (\ref{eq4-01}) is a cotorsion resolution of $M \otimes N$. Hence, $\cd(M \otimes N) = n$.

Similarly, the proof for (2) is analogous.
\end{proof}

\begin{proposition}
Let $M$ be an $R$-module and $N$ be an $S$-module. If $\cd(M) = 1$, $\cd(N) = 1$, then the following statements hold:
\begin{itemize}
\item[\upshape{(1)}] $0 < \cd(M \otimes N) \leq 2$ and $M \otimes N$ has a cotorsion resolution
\begin{equation}
0 \rightarrow M \otimes N \xrightarrow{\text{id}_M \oplus \partial_0} C_0 \otimes D_1 \oplus C_1 \otimes D_0 \xrightarrow{\partial_0 - \text{id}_0} C_0 \otimes D_0 \xrightarrow{\text{id}_1 \otimes \partial_1} C_1 \otimes D_1 \rightarrow 0.
\end{equation}
\item[\upshape{(2)}] $\cd(M \otimes N) = 1$ if and only if $C_0 \otimes D_1 + C_1 \otimes D_0$ is a cotorsion module.
\end{itemize}
\end{proposition}
\begin{proof}
(1) By Theorem \ref{t1}, $M \otimes N$ is not a cotorsion module. Therefore, $0 < \cd(M \otimes_k N)$. By $\cd(M) = 1$, $\cd(N) = 1$, we obtain the exact sequences
\begin{align}
0 \rightarrow M \xrightarrow{d_0} C^0 \xrightarrow{d_1} C^1 \rightarrow 0, \
0 \rightarrow N \xrightarrow{\partial^0} D^0 \xrightarrow{\partial^1} D^1 \rightarrow 0,
\end{align}
where $C^i$ and $D^j$ are cotorsion modules, $0 \leq i, j \leq 1$. Applying Proposition \ref{prop0}, in the category $\Mod(R \otimes S)$, we have a cotorsion resolution of $M \otimes N$:
\begin{equation}\label{eq4-1}
0 \rightarrow M \otimes N \xrightarrow{d_0 \oplus \partial^0} C^0 \otimes D^1 \oplus C^1 \otimes D^0 \xrightarrow{\partial^0 - d_0} C^0 \otimes D^0 \xrightarrow{d_1 \otimes \partial^1} C^1 \otimes D^1 \rightarrow 0.
\end{equation}
Therefore, $\cd(M \otimes N) \leq 2$.

(2) By the exactness of (\ref{eq4-1}), we obtain the short exact sequence
$$0 \rightarrow M \otimes N \xrightarrow{d_0 \oplus \partial^0} C^0 \otimes D^1 \oplus C^1 \otimes D^0 \rightarrow \Ker(d_1 \otimes \partial^1) \rightarrow 0,$$
where $\Ker(d_1 \otimes \partial^1) = C^0 \otimes D^1 + C^1 \otimes D^0$. By Proposition \ref{p5}, we have the desired conclusion.
\end{proof}
\begin{proposition}
Let $M$ be an $R$-module and $N$ be an $S$-module. If $\cd(M) = m$, $\cd(N) = n$, where $m, n \geq 1$ are positive integers, then $\cd(M \otimes N) \leq m + n$.
\end{proposition}
\begin{proof}
By assumption that $\cd(M) = m$, Note \ref{remark2-1} provides the shortest cotorsion resolution of $M$:
$$0 \rightarrow M \rightarrow C^0 \rightarrow C^1 \rightarrow \cdots \rightarrow C^m \rightarrow 0,$$
where $C^i$ is cotorsion modules, $0 \leq i \leq m$. Similarly, for $N$, we have the shortest cotorsion resolution
$$0 \rightarrow N \rightarrow D^0 \rightarrow D^1 \rightarrow \cdots \rightarrow D^n \rightarrow 0,$$
where $D^j$ is cotorsion modules, $0 \leq j \leq n$. Then we can obtain the following commutative diagram with exact rows and columns, where each column forms a cotorsion resolution:
$$\xymatrix@C=1.5em@R=2em{
  & 0 \ar[d]  & 0 \ar[d] & 0 \ar[d]  & 0 \ar[d]  & 0\ar[d] &0\ar[d] &\\
  0  \ar[r] &M\otimes N \ar[d] \ar[r]^{\lambda_{0}} &  C^0\otimes N \ar[d]\ar[r]^{\lambda_{1}}& C^1\otimes N \ar[d]\ar[r]^{\lambda_{2}}&\cdots\ar[d]\ar[r]&C^{m-1}\otimes N \ar[d]\ar[r]^{\lambda_{m}}& C^m\otimes N \ar[r]\ar[d]&0\\
 0  \ar[r] &M\otimes D^0 \ar[d] \ar[r] &  C^0\otimes D^0 \ar[d]\ar[r]& C^1\otimes D^0 \ar[d]\ar[r]&\cdots\ar[r]\ar[d]&C^{m-1}\otimes D^0 \ar[d]\ar[r]& C^n\otimes D^0  \ar[r]\ar[d]&0\\
 &\vdots\ar[d]&\vdots\ar[d]&\vdots\ar[d]&\vdots\ar[d]&\vdots\ar[d]&\vdots\ar[d]& \\
  0  \ar[r] &M\otimes D^{n-1} \ar[d] \ar[r] &  C^0\otimes D^{n-1} \ar[d]\ar[r]& C^1\otimes D^{n-1} \ar[d]\ar[r]&\cdots\ar[r]\ar[d]&C^{m-1}\otimes D^0 \ar[d]\ar[r]& C^m\otimes D^{n-1}  \ar[r]\ar[d]&0\\
  0  \ar[r] &M\otimes D^{n} \ar[d]\ar[r] &  C^0\otimes D^{n}\ar[d] \ar[r]& C^1\otimes D^{n} \ar[d]\ar[r]&\cdots\ar[r]\ar[d]&C^{m-1}\otimes \ar[d]D^n\ar[r]& C^m\otimes D^{n} \ar[d] \ar[r]&0\\
 & 0  & 0  & 0   & 0   & 0&0 &  }$$
From the second column, starting from $M \otimes D^0$, each  columns forms a cotorsion resolution of $C^i \otimes N$, $0 \leq i \leq m$. Therefore, $\cd(C^i \otimes N) \leq n$, $0 \leq i \leq m$. By observing the first row and according to Proposition \ref{p6}, we have
 \begin{align*}
    \cd(\Ker \lambda_{m}) & \leq \sup\{\cd(C^{m-1}\otimes N),\cd(C^m\otimes N)+1\}\leq n+1 \\
    \cd(\Ker \lambda_{m-1}) & \leq \sup\{\cd(C^{m-2}\otimes N),\cd(\Ker \lambda_{m})+1\}\leq n+2\\
    &\vdots\\
    \cd(\Ker \lambda_{2}) & \leq \sup\{\cd(C^3\otimes N),\cd(\Ker \lambda_{3})+1\}\leq n+m-1\\
    \cd(M\otimes N)&\leq \sup\{\cd(C^0\otimes N),\cd(\Ker \lambda_{2})+1\}\leq n+m
  \end{align*}
Thus, $\cd(M \otimes N) \leq m + n$.
\end{proof}
\begin{theorem}\label{thm2}
Let $M$ be an $R$-module and $N$ be an $S$-module. The following statements are equivalent:
\begin{itemize}
\item[\upshape{(P1)}] If $\cd(M) = 1$, $\cd(N) = 1$, then $\cd(M \otimes N) = 2$.
\item[\upshape{(P2)}] If $\cd(M) = m$, $\cd(N) = n$, where $m, n \geq 1$ are positive integers, then $\cd(M \otimes N) = m + n$.
\end{itemize}
\end{theorem}
\begin{proof}
$(P2) \Rightarrow (P1)$. The conclusion is obvious.

$(P1) \Rightarrow (P2)$. Using mathematical induction, we konw thst (1) is ture.Without loss of generality, we assume that $m \geq n \geq 1$. Since every module admits a cotorsion preenvelope, there exists a monomorphism  $\alpha: N \rightarrow C(N)$, where $C(N)$ is a cotorsion module, and injective module also is cotorsion moudule, so $\alpha$ is a monomorphism. Therefore, there exists a short exact sequence
$$0 \rightarrow N \rightarrow C(N) \rightarrow K \rightarrow 0,$$
where $\cd(K) = n - 1 > 0$. Applying the functor $M \otimes -$ to this short exact sequence, we obtain the short exact sequence
$$0 \rightarrow M \otimes N \rightarrow M \otimes C(N) \rightarrow M \otimes K \rightarrow 0.$$
By to Proposition \ref{prop4-0}, $\cd(M \otimes C(N)) = m$. From the induction hypothesis,
$$\cd(M \otimes K) = m + (n - 1) = m + n - 1 > 1 + n > m = \cd(M \otimes C(N))$$
Therefore, by Proposition \ref{prop2-7},
$$\cd(M \otimes N) = \cd(M \otimes K) + 1 = (m + n - 1) + 1 = m + n$$
\end{proof}
Consequently, we have the following result.
\begin{corollary}
If the $k$-algebras $R$ and $S$ satisfy property \upshape{(P1)} in Theorem~\ref{thm2}, then $l.cot.D(R \otimes S) \geq l.cot.D(R) + l.cot.D(S)$.
\end{corollary}
\bibliographystyle{amsplain}

\end{document}